\documentclass[11pt]{article}
\usepackage{amsfonts,amssymb,amsmath,amsthm}
\usepackage{mathrsfs}
\usepackage{fullpage}
\usepackage{enumitem}
\usepackage{graphicx}
\usepackage{tikz}

\newtheorem{theorem}{Theorem}
\newtheorem{corollary}[theorem]{Corollary}
\newtheorem{lemma}[theorem]{Lemma}
\newtheorem{proposition}[theorem]{Proposition}

\newtheorem{conjecture}[theorem]{Conjecture}

\newtheorem{case}{Case}

\newcommand{\floor}[1]{\left\lfloor#1\right\rfloor}

\newcommand{\st}{\colon\,}

\DeclareMathOperator{\squareop}{\mathop\square}

\begin{document}

\author{Jordan Almeter\thanks{College of William and Mary}, Samet Demircan\thanks{West Virginia University}, Andrew Kallmeyer\thanks{Miami University}, Kevin G. Milans\thanks{West Virginia University}, Robert Winslow\thanks{University of Kansas}}
\title{The 2-Ranking Numbers of Graphs}
\date{\today}

\newcommand{\rankarg}[1]{\chi_{#1}}
\newcommand{\krank}{\rankarg{k}}
\newcommand{\trank}{\rankarg{2}}
\newcommand{\infrank}{\rankarg{\infty}}
\newcommand{\chis}{\chi_{s}}
\newcommand{\F}{\mathbb{F}}
\newcommand{\dist}{\mathrm{dist}}

\newcommand{\ZZ}{\mathbb{Z}}
\newcommand{\PP}{\mathbb{P}}
\maketitle
\begin{abstract}
In a graph whose vertices are assigned integer ranks, a path is \emph{well-ranked} if the endpoints have distinct ranks or some interior point has a higher rank than the endpoints. A \emph{ranking} is an assignment of ranks such that all nontrivial paths are well-ranked.  A \emph{$k$-ranking} is a relaxation in which all nontrivial paths of length at most $k$ are well-ranked. The \emph{$k$-ranking number} of a graph $G$, denoted by $\krank(G)$, is the minimum $t$ such that there is a $k$-ranking of $G$ using ranks in $\{1,\ldots,t\}$. 

For the $n$-dimensional cube $Q_n$, we prove that $\trank(Q_n) = n+1$.  As a corollary, we improve the bounds on the star chromatic number of products of cycles when each cycle has length divisible by $4$.  We show that $\Omega(n\log m) \le \trank(K_m \squareop K_n) \le O(nm^{\log_2(3)-1})$ when $m\le n$ and obtain $\trank(K_m\squareop K_n)$ asymptotically in $n$ when $m$ is constant.  We prove that $\trank(G) \le 7$ when $G$ is subcubic, and we also prove the existence of a graph $G$ with maximum degree $k$ and $\trank(G) \ge \Omega(k^2/\log(k))$.
\end{abstract}

\section{Introduction}

A path consisting of a single vertex is \emph{trivial}; paths with positive length are \emph{nontrivial}.  In a graph whose vertices are assigned integer ranks, a path is \emph{well-ranked} if its endpoints have distinct ranks or some interior vertex has a higher rank than the endpoints.  A \emph{ranking} of a graph $G$ is an assignment of ranks to $V(G)$ such that every nontrivial path is well-ranked.  Graph rankings have arisen in mathematics and computer science; see the section on rankings in Gallian's dynamic survey~\cite{GallianSurvey} for a summary of results and background.  A \emph{$k$-ranking} is a relaxation under which each nontrivial path of length at most $k$ is well-ranked.  The \emph{$k$-ranking number} of $G$, denoted $\krank(G)$, is the minimum number of ranks in a $k$-ranking of $G$.  The \emph{ranking number} of $G$, denoted $\infrank(G)$, is the minimum $t$ such that $G$ has a ranking using ranks in $\{1,\ldots,t\}$.  We note that a $1$-ranking requires the endpoints of each edge to be assigned distinct colors, and so a $1$-ranking is a proper coloring.  It follows that $\rankarg{1}(G) = \chi(G)$, where $\chi(G)$ denotes the chromatic number of $G$.

The notion of a $2$-ranking was introduced by Karpas, Neiman, and Smorodinsky~\cite{KarpasUSColoring}, who used the term \emph{unique-superior coloring}.  We use the term \emph{$2$-ranking} to emphasize the connection with graph ranking.  In our terminology,  Karpas, Neiman, and Smorodinsky proved that the maximum, over all $n$-vertex trees $T$, of $\trank(T)$ is $\Theta(\frac{\log n}{\log\log n})$.  Trees are $K_3$-minor-free; it turns out that the $k$-ranking number of a graph grows at most logarithmically when some minor is excluded.  Specifically, Karpas, Neiman, and Smorodinsky show that for each graph $H$, there is a constant $s$ such that each $n$-vertex $H$-minor-free graph $G$ satisfies $\krank(G)\le s(k+1)\log n$.  A graph $G$ is \emph{$d$-degenerate} if each subgraph of $G$ has a vertex of degree at most $d$.  They also prove that each $n$-vertex $d$-degenerate graph $G$ satisfies $\trank(G)\le d(4\sqrt{n} + 1)$ and construct $n$-vertex $2$-degenerate graphs $G$ with $\trank(G) > n^{1/3}$.    

Our problems are also motivated by \emph{star colorings}, in which vertices are properly colored and every pair of color classes induces a star forest.  The \emph{star chromatic number} of $G$, denoted $\chis(G)$, is the minimum number of colors in a star coloring of $G$.  In a $2$-ranking of $G$, every pair of ranks induces a star forest in which the centers are all of the higher rank.  It follows that $\chis(G) \le \trank(G)$ for each graph $G$.  

One useful strategy to construct $k$-rankings is to assign distinct ranks to vertices that are close together.  The \emph{$k$-distance power} of $G$, denoted $G^k$, is the graph with vertex set $V(G)$ where $uv\in E(G^k)$ if and only if the distance between $u$ and $v$ in $G$ is at most $k$.  It is clear that $\krank(G) \le \chi(G^k)$.  Summarizing the relationship between our notions of graph colorings, we have
\[ \chi(G) \le \chis(G) \le \trank(G) \le \chi(G^2)\]
for each graph $G$.  In the above chain, it is not possible to bound any of the parameters as a function of its predecessor.  Indeed, the star shows that $\chi(G^2)$ can be arbitrarily large even when $\trank(G)=2$.  Every tree $T$ satisfies $\chis(T)\le 3$, but Karpas, Neiman, and Smorodinsky~\cite{KarpasUSColoring} show that the maximum, over all $n$-vertex trees, of $\trank(T)$ is $\Theta(\frac{\log n}{\log\log n})$.  For each integer $n$, we define $[n]= \{1,\ldots,n\}$.

\section{The hypercube}

The $d$-dimensional cube, denoted $Q_d$, is the graph with vertex set $\{0,1\}^d$ where $u$ and $v$ are adjacent if $u$ and $v$ differ in exactly one coordinate.  We prove that $\trank(Q_d) = d+1$.  The lower bound follows from a useful proposition.  A graph is $k$-degenerate if every subgraph contains a vertex of degree at most $k$.  The \emph{degeneracy} of a graph $G$ is the minimum integer $k$ such that $G$ is $k$-degenerate.

\begin{proposition}\label{prop:lowerbound}
If $G$ is a graph with degeneracy $k$, then $\trank(G) \geq k+1$.
\end{proposition}
\begin{proof}
Since $G$ is not $(k-1)$-degenerate, $G$ contains a subgraph $H$ with minimum degree at least $k$.  Consider a $2$-ranking of $G$, and let $v$ be a vertex of minimum rank in $H$.  The ranks of the neighbors of $v$ in $H$ are distinct, and the rank of $v$ differs from all of these.  It follows that $\trank(G) \ge k+1$.
\end{proof}

Since $Q_d$ is $d$-regular, it follows that $\trank(Q_d) \ge d+1$.  Wan~\cite{wan1997near} proved that $\chi(Q_d^2) = d+1$ when $d=2^k - 1$ for some integer $k$, and it follows that $d+1\le \trank(Q_d) \le \chi(Q_d^2) = d+1$ in this case.  Each color class in a proper coloring of $Q_d^2$ has size at most $\floor{2^d/(d+1)}$, and it follows that $\chi(Q_d^2) \ge 2^d/\floor{2^d/(d+1)}$.  Therefore $\chi(Q_d^2) > d+1$ when $d$ does not have the form $2^k - 1$.  Nonetheless, we show that $\trank(Q_d) = d+1$ for all $d$.  Although determining the exact value of $\chi(Q_d^2)$ remains open, {\"O}sterg{\aa}rd~\cite{distance2cube} proved that $\chi(Q_d^2) = (1+o(1))d$.  

We view the vertex set of $Q_d$ as $\F_2^d$, the $d$-dimensional vector space over the finite field $\F_2$ with $2$ elements.  For $u\in \F_2^d$, we define the \emph{support} of $u$ to be the set of coordinates in $[d]$ where $u$ has value $1$.  The \emph{weight} of $u$, denoted $w(u)$, is the size of the support of $u$.  Note that for all vertices $u,v\in \F_2^n$, we have that $\dist(u,v)= w(u-v)$, where $\dist(u,v)$ is the length of a shortest path from $u$ to $v$ in $Q_d$.  For integers $i$ and $j$, we use $[i,j]$ to denote the interval $\{i,i+1,\ldots,j\}$.

\begin{theorem}\label{thm:trank-cube}
$\trank(Q_d) = d+1$.
\end{theorem}

\begin{proof}
As we have seen, $\trank(Q_d) \ge d + 1$.  We prove the upper bound by induction on $d$.  The result for $d\in \{0,1\}$ is trivial, since the vertices may be assigned distinct ranks in the interval $[0, d]$.  Suppose that $d\ge 2$, and express $d$ as $t+2^k$ where $k \ge 1$ and $0\le t \le 2^k - 1$.  Given $u\in \F_2^d$, we let $u^-$ be the vector in $\F_2^t$ consisting of the first $t$ coordinates of $u$ and we let $u^+$ be the $2^k$-dimensional vector consisting of the remaining coordinates.  If $w(u^+)$ is even, then we set the rank of $u$ equal to the rank in $[0,t]$ assigned to $u^-$ inductively.  

If $w(u^+)$ is odd, then we assign $u$ a rank in the interval $[t+1, d]$ as follows.  Let $A$ be a $(k\times d)$-matrix whose first $t$ columns are distinct, nonzero vectors in $\F_2^k$ and whose last $2^k$ columns form a permutation of $\F_2^k$.  Let $\phi\st \F_2^k \to [t+1, d]$ be a bijection.  When $w(u^+)$ is odd, we set the rank of $u$ to be $\phi(Au)$.  Ranks in the range $[0,t]$ are \emph{low}, and ranks in the range $[t+1, d]$ are \emph{high}.  

We show that this assignment is a $2$-ranking.  Let $P$ be a $uv$-path of length $1$ or $2$.  Note that $w(u-v)$ equals the length of $P$.  We prove that $P$ is well-ranked by examining several cases.

\begin{case}
The support of $u-v$ is contained in the first $t$ coordinates.
\end{case}
We have that $u^+ = v^+$.  If $w(u^+)$ and $w(v^+)$ are even, then the vertices of $P$ are colored inductively and so $P$ is well-ranked by induction.  Otherwise both $w(u^+)$ and $w(v^+)$ are odd, and so $u$ is assigned rank $\phi(Au)$ and $v$ is assigned rank $\phi(Av)$.  Since $w(u-v)\in \{1,2\}$, it follows that $A(u-v)$ is the sum of one or two of the first $t$ columns of $A$.  Since these columns are nonzero and distinct, we have that $A(u-v) \ne 0$ and it follows that $u$ and $v$ are assigned different ranks.  Therefore $P$ is well-ranked.

\begin{case}
The support of $u-v$ is contained in the last $2^k$ coordinates and $\dist(u,v) = 1$.
\end{case}
Since $w(u-v) = \dist(u,v) = 1$, it follows that $w(u^+)$ and $w(v^+)$ have opposite parity, implying that one of $\{u,v\}$ is assigned a high rank and the other is assigned a low rank.  

\begin{case}
The support of $u-v$ is contained in the last $2^k$ coordinates and $\dist(u,v) = 2$.
\end{case}
We have that $w(u^+)$ and $w(v^+)$ have the same parity.  Let $x$ be the internal vertex on $P$, and note that $w(x^+)$ has opposite parity.  If both $w(u^+)$ and $w(v^+)$ are even and $w(x^+)$ is odd, then the endpoints $u$ and $v$ are assigned low rank while $x$ is assigned high rank, and so $P$ is well-ranked.  If both $w(u^+)$ and $w(v^+)$ are odd, then $u$ has rank $\phi(Au)$ and $v$ has rank $\phi(Av)$.  Since $w(u-v) = \dist(u,v) = 2$ and the support of $u-v$ is contained in the last $2^k$ coordinates, it follows that $A(u-v)$ is the sum of two columns from the last $2^k$ columns in $A$.  Since these are distinct, it follows that $A(u-v) \ne 0$.  Therefore $Au\ne Av$, and so $P$ is well-ranked.

\begin{case}
The support of $u-v$ intersects both the first $t$ coordinates and the last $2^k$ coordinates.
\end{case}
We have that $w(u^+)$ and $w(v^+)$ have opposite parity.  Therefore one of $\{u,v\}$ has high rank and the other has low rank.

In all cases, $P$ is well-ranked.  
\end{proof}

The $2$-ranking given in Theorem~\ref{thm:trank-cube} assigns the same low rank to $u$ and $v$ whenever $u^- = v^-$ and both $w(u^+)$ and $w(v^+)$ are even.  Consequently, when $d\ge 3$, many pairs of vertices at distance $2$ share a common rank.  When $d$ is one less than a power of two, a proper coloring of $Q_d^2$ is a $2$-ranking of $Q_d$ in which pairs of vertices at distance $2$ receive distinct ranks.  It follows that when $d\ge 3$ and $d$ has the form $2^k-1$, there are non-isomorphic optimal $2$-rankings of $Q_d$.  The situation when $d$ has the form $2^k$ may be different.  For $d\in\{1,2\}$, there is only one optimal $2$-ranking of $Q_d$ up to isomorphism.  We suspect that $Q_4$ has only one optimal $2$-ranking up to isomorphism.  Is it true that $Q_d$ has one optimal $2$-ranking up to isomorphism when $d$ is a power of two?

The \emph{cartesian product} of $G$ and $H$, denoted $G\squareop H$, is the graph with vertex set $V(G)\times V(H)$ where $(u,v)$ is adjacent to $(u',v')$ if and only if $u=u'$ and $vv'\in E(H)$ or $uu'\in E(G)$ and $v=v'$.

\begin{corollary}\label{cor:prodcycles}
If $G$ is the cartesian product of $d$ cycles, each of which has length divisible by $4$, then $\chis(G) \le \trank(G) = 2d + 1$.
\end{corollary}

\begin{proof}
Since $G$ has degeneracy $2d$, Proposition~\ref{prop:lowerbound} implies that $\trank(G)\ge 2d+1$.  Note that $Q_{2d}$ is the cartesian product of $d$ copies of $C_4$.  Viewing $V(Q_{2d})$ as $\ZZ_4^d$, let $f\st \ZZ_4^d\to [2d+1]$ be a $2$-ranking of $Q_{2d}$.  We use $f$ to color $G$.  Let $m_1, \ldots, m_d$ be the cycle lengths of the factors of $G$, and view $V(G)$ as $\{(x_1, \ldots, x_d)\st x_i \in \ZZ_{m_i}\}$.  For $x\in V(G)$, let $x'$ be the vertex in $Q_{2d}$ obtained from $x$ by reducing each coordinate of $x$ modulo $4$.  We assign $x\in V(G)$ the rank $f(x')$.  Since each path $P$ in $G$ of length at most $3$ maps to a path in $Q_{2d}$ of the same length whose vertices are assigned the same ranks as in $G$, it follows that $G$ inherits the $2$-ranking of $Q_{2d}$.
\end{proof}

Let $G$ be the cartesian product of $d$ cycles.  Fertin, Raspaud, and Reed~\cite{fertinStar} proved that $d+2 \le \chis(G) \le 2d^2+d+1$, and improved the upper bound to $2d+1$ in the case that $2d+1$ divides the length of each factor cycle.  P\'or and Wood~\cite{PorWood} proved that $G$ admits a proper $(6d+O(\log d))$-coloring in which each pair of color classes induces a matching and isolated vertices; their result directly implies that $\chis(G) \le 6d + O(\log d)$.  Corollary~\ref{cor:prodcycles} extends the divisibility conditions under which it is known that $\chis(G) \le 2d+1$.

\section{Cartesian products of complete graphs}

Determining the $2$-ranking number of $K_m \squareop K_n$ is an interesting problem.  For each fixed $m$, we obtain $\trank(K_m \squareop
K_n)$ asymptotically.  When $m=n$, our bounds are far apart.  A $2$-ranking of $K_m \squareop K_n$ can be viewed as an $(m\times n)$-matrix $A$ such that $A(i,j)$ is the rank of $(u_i, v_j) \in V(K_m \squareop K_n)$.  The condition that paths of length $1$ are well-ranked is equivalent to the rows and columns of $A$ having distinct entries.  The condition that paths of length $2$ are well-ranked is equivalent to the property that $A(i,j) = A(i',j')$ implies that the opposite corners $A(i,j')$ and $A(i',j)$ are larger than $A(i,j)$ and $A(i',j')$.

For positive integers $a,b,c,d$, our first result obtains a $2$-ranking of $K_{ac}\squareop K_{bd}$ from $2$-rankings of $K_a\squareop K_b$ and $K_c\squareop K_d$.

\begin{proposition}\label{prop:cart-comp}
$\trank(K_{ac}\squareop K_{bd})\le \trank(K_a\squareop K_b)\cdot \trank(K_c\squareop K_d)$.
\end{proposition}
\begin{proof}
Let $k = \trank(K_a\squareop K_b)$ and $\ell = \trank(K_c\squareop K_d)$.  Let $A$ be an $(a\times b)$-matrix with entries in $\{0, \ldots, k-1\}$ encoding an optimal $2$-ranking of $K_a\squareop K_b$, and let $B$ be an $(c\times d)$-matrix with entries in $\{0, \ldots, \ell-1\}$ encoding an optimal $2$-ranking of $K_c\squareop K_d$.  We use block operations to construct a $2$-ranking of $K_{ac}\squareop K_{bd}$.  

Let $C$ be the $(ac\times bd)$-matrix obtained from $A$ and $B$ by replacing each entry $A(i,j)$ in $A$ with the $(c\times d)$-matrix $\ell A(i,j) + B$.  It is easy to see that $C$ encodes a $2$-ranking of $K_{ac}\squareop K_{bd}$.  Since the entries in $C$ belong to $\{0,\ldots, k\ell - 1\}$, we have that $\trank(K_{ac}\squareop K_{bd}) \le k\ell$.
\end{proof}

Proposition~\ref{prop:cart-comp} may be iterated to obtain upper bounds on $\trank(K_n \squareop K_n)$.

\begin{corollary}\label{cor:square-upper}
If $m$ and $n$ are powers of 2 with $m\le n$, then $\trank(K_m\squareop K_n)\le nm^{\log_2(3) - 1} \approx nm^{0.585}$.
\end{corollary}
\begin{proof}
Observe that 
$\left[\begin{array}{cc}
1 & 0\\
0 & 2\\
\end{array}
\right]$
is a $2$-ranking witnessing that $\trank(K_2\squareop K_2) \le 3$.  If $m=1$, then $\trank(K_m\squareop K_n) = n$, and so the bound holds.  Otherwise, by Proposition~\ref{prop:cart-comp} and induction, we have that $\trank(K_m\squareop K_n) \le \trank(K_{m/2} \squareop K_{n/2})\cdot\trank(K_2\squareop K_2) \le \frac{n}{2} \left(\frac{m}{2}\right)^{\log_2(3)-1} \cdot 3  = nm^{\log_2(3)-1}$.
\end{proof}

When $m$ and $n$ are not powers of two, we may apply Corollary~\ref{cor:square-upper} to $K_{m'}\squareop K_{n'}$ where $m'$ and $n'$ are the least powers of two larger than $m$ and $n$, respectively.  Since $m' < 2m$ and $n' < 2n$, this gives $\trank(K_m\squareop K_n) < 3nm^{\log_2(3) - 1}$ for general $m$ and $n$.  To prove a lower bound on $\trank(K_m \squareop K_n)$, we restrict the number of times that certain ranks can appear.

\begin{lemma}\label{lem:rank-restrict}
In a $2$-ranking of $K_m\squareop K_n$, each column of height $m$ contains $k$ ranks which are assigned to at most $k$ vertices for $1\le k\le m$.
\end{lemma}
\begin{proof}
Let $A$ be an $(m\times n)$-matrix representing a $2$-ranking of $K_m\squareop K_n$, and let $x$ be the $j$th column in $A$.  Let $R$ be the set of rows containing the $k$ highest ranks in $x$, and let $S = \{A(i,j)\st i\in R\}$.  We claim that each rank in $S$ appears only in rows in $R$.  Since each rank appears at most once in each row, it then follows that each of the $k$ ranks in $S$ is assigned to at most $k$ vertices.

Suppose that $A(i,j) = A(i',j')$ where $i\in R$.  Since $A$ is a $2$-ranking, it must be that $A(i',j) > A(i,j)$, which implies that $A(i',j)$ is among the $k$ highest ranks in $x$.  Therefore $i'\in R$ also.  It follows that each rank in $S$ appears only in rows in $R$.
\end{proof}

Lemma~\ref{lem:rank-restrict} forces a nontrivial number of ranks in a $2$-ranking of $K_m\squareop K_n$.
\begin{theorem}\label{thm:lower}
We have $\trank(K_m\squareop K_n) \ge nH_m$, where $H_m$ is the Harmonic series $1 + \frac{1}{2} + \frac{1}{3} + \cdots + \frac{1}{m}$.
\end{theorem}
\begin{proof}
Let $A$ be an $(m\times n)$-matrix encoding an optimal $2$-ranking of $K_m\squareop K_n$, and let $a_k$ be the number of ranks that $A$ assigns to exactly $k$ vertices.  Note that $\trank(K_m\squareop K_n) = \sum_{k=1}^m a_k$.  We claim that for $1\le k\le m$, we have that $\sum_{i=1}^k ia_i \ge kn$.  Indeed, $\sum_{i=1}^k ia_i$ counts the number of vertices in $K_m\squareop K_n$ whose ranks appear at most $k$ times in $A$.  By Lemma~\ref{lem:rank-restrict}, for $1\le k\le m$, each of the $n$ columns in $A$ is associated with $k$ such vertices.  Therefore $\sum_{i=1}^k ia_i \ge kn$ as claimed.  

Let $a_1, \ldots, a_m$ minimize $\sum_{i=1}^m a_i$ subject to the conditions $\sum_{i=1}^k ia_i \ge kn$ for $k\in[m]$.  We claim that in each constraint, equality holds.  Indeed, if $k$ is the least integer such that $\sum_{i=1}^k ia_i > kn$, then we may reduce $a_k$ by a positive $\varepsilon$ while still satisfying the constraints $\sum_{i=1}^\ell ia_i \ge \ell n$ for $1\le \ell \le k$.  If we also increase $a_{k+1}$ by $\frac{k}{k+1} \varepsilon$, then all constraints are satisfied, but we have reduced $\sum_{i=1}^m a_i$ by $\frac{1}{k+1}\varepsilon$, contradicting the minimality of $\sum_{i=1}^m a_i$.  

Since equality holds in all constraints, we conclude $a_k = n/k$ for each $k$ and $\sum_{i=1}^m a_i = nH_m$.
\end{proof}

When $n \ge m!$, Theorem~\ref{thm:lower} gives the correct order of growth of $\trank(K_m\squareop K_n)$.  In fact, equality holds when $m! \mid n$.

\begin{theorem}\label{thm:equality}
If $m!\mid n$, then $\trank(K_m\squareop K_n) = nH_m$.
\end{theorem}
\begin{proof}
Theorem~\ref{thm:lower} gives the lower bound.  We claim that it suffices to prove $\trank(K_m\squareop K_{m!})\le (m!)H_m$.  Indeed, with $n=tm!$, the general case would then follow from Proposition~\ref{prop:cart-comp}, since $\trank(K_m\squareop K_n) \le \trank(K_m\squareop K_{m!}) \cdot \trank(K_1\squareop K_t) = (m!)H_m\cdot t = nH_m$.

We prove that $\trank(K_m\squareop K_{m!}) = (m!)H_m$ by induction on $m$.  For $m=1$, the statement is trivial.  Suppose that $m\ge 2$ and let $A'$ be an $((m-1)\times (m-1)!)$-matrix encoding an optimal $2$-ranking of $K_{m-1}\squareop K_{(m-1)!}$.  By shifting the ranks appropriately, let $A'_1, \ldots, A'_m$ be copies of $A'$ that use disjoint intervals of ranks, starting with rank $(m-1)!+1$.  The ranks appearing in $A'_1, \ldots, A'_m$ are \emph{high}, and the ranks in $[(m-1)!]$ are \emph{low}.  

We construct an $(m\times m!)$-matrix $A$ encoding a $2$-ranking of $K_m\squareop K_{m!}$ as follows.  Let $M_i$ be an $(m\times (m-1)!)$-matrix such that deleting the $i$th row of $M_i$ gives $A'_i$ and whose $i$th row contains each low rank.  Let $A = [M_1 \cdots M_m]$.  The rows and columns of $A$ have distinct entries.  Suppose that $A(i,j) = A(i',j')$.  If $A(i,j)$ and $A(i',j')$ are both low ranks, then columns $j$ and $j'$ belong to distinct blocks of $A$ and so $A(i',j)$ and $A(i,j')$ are both high ranks.  If $A(i,j)$ and $A(i',j')$ are both high ranks, then columns $j$ and $j'$ belong to the same block of $A$ and so the opposite corners have higher rank by induction.  It follows that $A$ is a $2$-ranking.  Since $A$ uses $(m-1)!$ low ranks and $m\cdot [(m-1)!H_{m-1}]$ high ranks, we have that $\trank(K_m\squareop K_{m!}) \le (m-1)! + m!H_{m-1} = m!(1/m + H_{m-1}) = m!H_m$. 
\end{proof}

Using that $H_m = (1+o(1))\ln m$, we obtain an asymptotic formula for $\trank(K_m\squareop K_n)$ when $m$ is constant.

\begin{corollary}\label{cor:asymptotic}
For each positive integer $m$, we have that $\trank(K_m\squareop K_n) = (1+o(1))n\ln m$ as $n\to\infty$.
\end{corollary}
\begin{proof}
The lower bound follows immediately from Theorem~\ref{thm:lower}.  For the upper bound, let $n'$ be the least multiple of $m!$ that is at least $n$.  By Theorem~\ref{thm:equality}, we have $\trank(K_m\squareop K_n) \le \trank(K_m\squareop K_{n'}) = n'H_m \le (n+m!)H_m = (1+(m!)/n)\cdot nH_m = (1+o(1))n\ln m$.
\end{proof}

In the diagonal case, our bounds are far apart.  Combining Theorem~\ref{thm:lower} and Corollary~\ref{cor:square-upper} gives $\Omega(n\log n) \le \trank (K_n\squareop K_n) \le O(n^{\log_2 3})$.  What is the order of growth of $\trank(K_n\squareop K_n)$?

\section{The $2$-ranking number of graphs with maximum degree $k$}

Let $G$ be a graph with $\Delta(G)=k$, where $\Delta(G)$ is the maximum degree of $G$.  Since $\chis(G) \le \trank(G)\le \chi(G^2) \le \Delta(G^2) + 1\le k^2 + 1$, it is interesting to ask for the maximum of $\chis(G)$ and $\trank(G)$ over all graphs $G$ with maximum degree at most $k$.  Fertin, Raspaud, Reed~\cite{fertinStar} proved that the maximum of $\chis(G)$ over all graphs with maximum degree at most $k$ is at least $\Omega(\frac{k^{3/2}}{(\log k)^{1/2}})$ and is at most $O(k^{3/2})$.  We make slight modifications to their probabilistic construction to show that the maximum of $\trank(G)$ over all graphs with maximum degree $k$ is at least $\Omega(k^2/\log k)$.

\begin{theorem}
For each $k$, there exists a graph $G$ with $\Delta(G) \le k$ and $\trank(G) \ge \Omega(k^2/\log k)$.
\end{theorem}

\begin{proof}
Choose $n$ so that $n$ is even and $2np\le k$, where $p=c(\log n/n)^{1/2}$ for some constant $c$ to be chosen later.  Since we may assume that $k$ is sufficiently large, we may assume that $n$ is also sufficiently large.  Let $G$ be a random graph chosen from $G(n,p)$.  Each vertex in $G$ has expected degree $(n-1)p$, and it is well known (see, for example,~\cite{bollobas1998random}) that $\PP(\Delta(G) \le 2np) \to 1$ as $n\to \infty$.  For each function $f\st V(G) \to [n/2]$, let $A_f$ be the bad event that $f$ is a $2$-ranking of $G$.  Applying the union bound, we have that $\PP(\trank(G)\leq \frac{n}{2}) = \PP(\bigcup_f A_f) \le \sum_f \PP(A_f)$.

Fix a function $f\st V(G)\to [n/2]$.  Discarding one vertex from each rank class with an odd number of vertices, we may partition the remaining vertices into pairs $S_1, \ldots, S_\ell$ such that both vertices on $S_i$ have the same rank under $f$.  Since at most $n/2$ vertices are discarded, we have $\ell \ge (1/2)(n - n/2) = n/4$.  Index the pairs so that $i\le j$ implies that $f(u)\le f(v)$ when $u\in S_i$ and $v\in S_j$.  For each pair $\{S_i, S_j\}$ with $i<j$, the probability that $G$ contains some path $uwv$ such that $u,v\in S_j$ and $w\in  S_i$ is at least $p^2$.  If this happens, then $f$ is not a $2$-ranking since either $f(u) = f(v) = f(w)$ or $f(u) = f(v) > f(w)$.  

\begin{center}\includegraphics{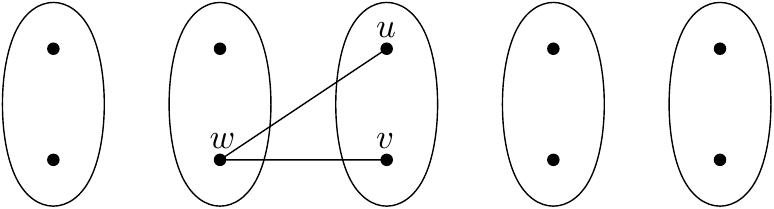}\end{center}
Since the paths $uwv$ form an edge-disjoint family as we range over the pairs $\{S_i,S_j\}$, it follows that the pairs $\{S_i, S_j\}$ give independent chances for $A_f$ to fail.  It follows that $\PP(A_f) \leq (1-p^2)^{\binom{n/4}{2}}$. For sufficiently large $n$, it follows that 
\[
\PP(\trank(G)\le \frac{n}{2}) \le \sum_f \PP(A_f) \le (n/2)^n (1-p^2)^{\binom{n/4}{2}} \le (n/2)^n e^{-\frac{p^2n^2}{33}} = \left(\frac{n}{2n^{c^2/33}}\right)^n.
\]
With $c=6$, we have that $\PP(\trank(G)\le \frac{n}{2}) \to 0$ as $n\to\infty$.  It follows that with probability tending to $1$, we have that $\Delta(G)\le k$ and $\trank(G) > n/2 \ge c' \frac{k^2}{\ln k}$ for some positive constant $c'$. 
\end{proof}

\section{The $2$-ranking number of subcubic graphs}

\newcommand{\diam}{\mathrm{diam}}

A graph $G$ is \emph{subcubic} if $\Delta(G)\le 3$.  The \emph{star list chromatic number} of $G$, denoted $\chi_s^\ell(G)$, is the minimum integer $t$ such that if each vertex $v$ in $G$ is assigned a list $L(v)$ of $t$ colors, there is a star coloring of $G$ in which each vertex $v$ receives a color from its list $L(v)$.  Albertson, Chappell, Kierstead, K{\"u}ndgen, and Ramamurthi~\cite{Albertson} gave an elegant proof that every subcubic graph $G$ satisfies $\chi_s^\ell(G)\le 7$.  It follows that $\chis(G) \le \chi_s^\ell(G)\le 7$ when $G$ is subcubic.  Chen, Raspaud, and Wang~\cite{chen20136} proved that every subcubic graph $G$ satisfies $\chis(G)\le 6$.  

Let $G$ be the $3$-regular graph obtained from $C_8$ by joining vertices at distance $4$.  Fertin, Raspaud, and Reed~\cite{fertinStar} proved that $\chis(G)=6$, and it follows that the result of Chen, Raspaud, and Wang is best possible.  

Here, we show that $\trank(G)\le 7$ when $G$ is subcubic.  Since $\trank(G)\ge \chis(G)$ always, the example of Fertin, Raspaud, and Reed shows that our bound cannot be reduced by more than $1$ in the general case.  Nonetheless, we believe the bound can be improved by $2$ aside from a single exception; see Conjecture~\ref{conj:subcubic}.  

An \emph{independent set} in $G$ is a set of vertices that are pairwise nonadjacent.  We use $N_G(u)$ for the set of neighbors of $u$ in $G$ and, when $S\subseteq V(G)$, we use $G[S]$ for the subgraph of $G$ induced by $S$. Vertices $u$ and $v$ in a graph $G$ are \emph{antipodal} if $\dist(u,v) = \diam(G)$, where $\diam(G)$ is the maximum distance between a pair of vertices in $G$.

\newcommand{\barS}{\overline{S}}
\begin{theorem}
If $G$ is subcubic, then $\trank(G)\le 7$.
\end{theorem}
\begin{proof}
Let $G$ be a subcubic graph.  We may assume that $G$ is connected.  Let $S$ be a maximal independent subset of $V(G)$, and let $\barS = V(G) - S$.  Since $S$ is maximal, every vertex in $G$ is in $S$ or has a neighbor in $S$.  It follows that $|N_{G^2}(v) \cap \barS| \le 6$.  Indeed, if $u$ is a neighbor of $v$, then $u$ has at most $1$ other neighbor in $\barS$, or else $u$ would have $3$ neighbors in $\barS$, a contradiction.  

It follows that $\Delta(G^2[\barS]) \le 6$.  If $G^2[\barS]$ does not contain a copy of $K_7$, then by Brooks's theorem, $\chi(G^2[\barS])\le 6$.  Using a proper coloring of $G^2[\barS]$ with colors in $[6]$ and assigning rank $0$ to all vertices in $S$ gives a $2$-ranking of $G$.  Indeed, paths of length $1$ are well-ranked and $G$ contains no paths of length 2 joining vertices with nonzero ranks.

Since $G$ is connected, it follows that $G^2[\barS]$ is connected.  Since $G^2[\barS]$ is connected and has maximum degree at most $6$, if $G^2[\barS]$ contains a copy of $K_7$, then $G^2[\barS] = K_7$.  It follows that $\trank(G) \le 7$ unless $G^2[\barS] = K_7$. 

Suppose that $G^2[\barS] = K_7$.  This has several implications for the structure of $G[\barS]$.  First, we claim that every vertex in $G[\barS]$ has degree $0$ or degree $2$.  Since each vertex $u\in \barS$ has a neighbor in $S$, it follows that $u$ has at most $2$ neighbors in $G[\barS]$.  If the only neighbor of $u$ in $G[\barS]$ is $v$, then $v$ has at most $5$ neighbors in $G^2[\barS]$: at most $2$ from each of the neighbors of $v$ in $G$ besides $u$, and $u$ itself.  This contradicts that $G^2[\barS] = K_7$.

It follows that $G[\barS]$ is a $7$-vertex graph whose components are isolated vertices and cycles.  We claim that each cycle in $G[\barS]$ has length at least $5$.  Indeed, suppose that $v$ is in a cycle $C$ in $G^2[\barS]$ of length at most $4$, and let $u_1$ and $u_2$ be the neighbors of $v$ along $C$.  Each neighbor of $v$ in $G$ contributes at most $2$ neighbors of $v$ in $G^2[\barS]$.  Since $C$ has length at most $4$, the contributions of $u_1$ and $u_2$ have nonempty intersection.  It follows that $v$ has fewer than $6$ neighbors in $G^2[\barS]$, contradicting that $G^2[\barS] = K_7$.  

Suppose that $G[\barS]$ contains a 5-cycle $C$, and let $x$ and $y$ be the vertices in $G[\barS]$ that are not in $C$.  Let $u$ be a vertex in $C$.  Since $u$ is adjacent to $x$ and $y$ in $G^2[\barS]$, it must be that $u$ is adjacent in $G$ to a vertex $s_u\in S$ whose other two neighbors in $G$ are $x$ and $y$.  The vertices $\{s_u\st u\in V(C)\}$ have distinct neighborhoods of size $3$ and are therefore distinct.  This is not possible, since $x$ and $y$ would have $5$ neighbors in $G$.  Therefore $G[\barS]$ does not contain a 5-cycle.

Suppose that $G[\barS]$ contains a 7-cycle $C$, and let $u$ be a vertex in $C$.  Since $u$ is adjacent in $G^2[\barS]$ to the two vertices $x$ and $y$ at distance $3$ from $u$ in $C$, it must be that $u$ is adjacent in $G$ to a vertex $s_u$ such that $N_G(s_u) = \{u,x,y\}$.  Again, the vertices $\{s_u\st u\in V(C)\}$ have distinct neighborhoods of size $3$, and are therefore distinct.  This is impossible, since $x$ is adjacent in $G$ to $s_u$, $s_x$, and its two neighbors on $C$.  Therefore, $G[\barS]$ cannot contain a 7-cycle.

It follows that either $G[\barS]= C_6 + K_1$ or $G[\barS] = 7K_1$.  Suppose that $G[\barS]$ contains a $6$-cycle $C$ and let $x$ be the isolated vertex.  If $u$ is a vertex on $C$, then $u$ is adjacent in $G$ to a vertex $s_u$ whose neighbors are $u$, $x$, and the vertex on $C$ antipodal to $u$.  It follows that $G$ is the Petersen graph.  In the diagram below, vertices in $S$ are white and vertices in $\barS$ are black.
\begin{center}
\includegraphics[]{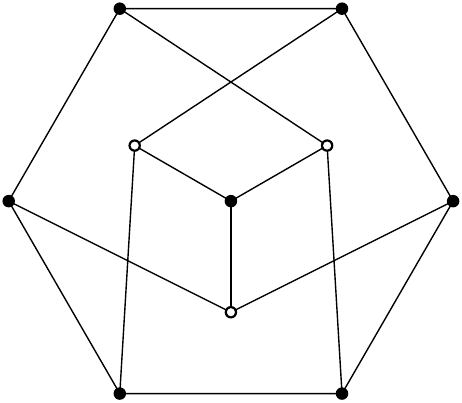}
\end{center}
Suppose that $G[\barS] = 7K_1$.  It follows that each vertex $u$ in $\barS$ is adjacent in $G$ to $3$ neighbors $v_1, v_2, v_3$ in $S$.  Moreover, we have $\bigcup_{i=1}^3 N_G(v_i) = \barS$ and $N_G(v_i) \cap N_G(v_j) = \{u\}$ for $i\ne j$.  Since $G$ is connected, we have that $G$ is a $3$-regular $(S,\barS)$-bigraph, and so $|S| = |\barS| = 7$.  It follows also that $G$ does not contain a copy of $C_4$, or else some vertex $u\in\barS$ would have neighbors $v_1$ and $v_2$ with $|N_G(v_1) \cap N_G(v_2)| \ge 2$.  Since $G$ is a bipartite 3-regular graph on 14 vertices with girth at least 6, it follows that $G$ is the Heawood graph.

As we have seen, if $G$ is subcubic, then $\trank(G)\le 7$, or $G$ is the Petersen graph, or $G$ is the Heawood graph.  If $G$ is the Petersen graph, then $G$ contains a maximal independent set $S$ of size $4$.  We may repeat the argument with $G[\barS]$ having $6$ vertices.  If $G$ is the Heawood graph, then a vertex $u$ and the four vertices antipodal to $u$ form a maximal independent set $S$ of size $5$.  We we may repeat the argument with $G[\barS]$ having $9$ vertices.
\end{proof}

Besides the example of Fertin, Raspaud, and Reed, we are not aware of another subcubic graph that requires $6$ ranks for a $2$-ranking.  Plausible candidates such as the Petersen graph and the Heawood graph admit $2$-rankings with only $5$ ranks.

\begin{conjecture}\label{conj:subcubic}
If $G$ is subcubic, then $\trank(G)\le 6$ and equality holds if and only if $G$ is the cubic graph obtained from $C_8$ by joining vertices at distance $4$.
\end{conjecture}

\section{The product of a triangle and a cycle}

Applied to the product of a pair of cycles, Corollary~\ref{cor:prodcycles} states that $\trank(C_m\squareop C_n) = 5$ when $m$ and $n$ are divisible by $4$.  In this section, we show that the $2$-ranking number of cycle products may depend on the parity of the lengths of the factors.  In particular, we show that for sufficiently large $n$, the $2$-ranking number of $C_3 \squareop C_n$ is $5$ when $n$ is even and $6$ when $n$ is odd.  We represent a $2$-ranking of $C_3\squareop C_n$ with a $(3\times n)$-array $A$ such that $A(i,j)$ is the rank of $(u_i,v_j)\in V(C_3\squareop C_n)$.

\begin{lemma}\label{lem:upbd6}
If $n\ge 24$, then $\trank(C_3 \squareop C_n) \le 6$.
\end{lemma}
\begin{proof}
Let $n=4q+r$ for integers $q$ and $r$ with $r\in\{0,1,2,3\}$.  Since $q\ge 6$, we have that $n=4(q-2r)+9r$, and so $n$ is a nonnegative integer combination of $4$ and $9$.  We give $2$-rankings of $C_3\squareop C_9$ and $C_3\squareop C_4$ that can be appended together to give a $2$-ranking of $C_3\squareop C_n$.

\begin{center}
\hfill
$
\begin{array}{ccccccccc}
2 & 4 & 0 & 3 & 1 & 0 & 4 & 0 & 5\\
0 & 5 & 1 & 0 & 5 & 2 & 0 & 1 & 3\\
1 & 3 & 2 & 4 & 0 & 3 & 5 & 2 & 4\\
\end{array}
$
\hfill
$
\begin{array}{cccc}
2 & 4 & 0 & 5\\
0 & 5 & 1 & 3\\
1 & 3 & 2 & 4\\
\end{array}
$
\hfill~
\end{center}

To see that these are $2$-rankings, observe that the vertices assigned rank $0$ form an independent set, and for each positive rank $t$, the vertices assigned rank $t$ are independent in $(C_3\squareop C_n)^2$.  Because both $2$-rankings agree on the first two columns and the last two columns, appending the arrays gives a $2$-ranking.
\end{proof}

Our upper bound improves for even $n$.

\begin{lemma}\label{lem:upbd5}
If $n$ is even and $n\ge 4$, then $\trank(C_3 \squareop C_n)\le 5$.
\end{lemma}
\begin{proof}
Let $n=4q + 6r$ for integers $q$ and $r$ with $r\in\{0,1\}$.  We give $2$-rankings of $C_3\squareop C_4$ and $C_3\squareop C_6$ which can be appended to give a $2$-ranking of $C_3\squareop C_n$.

\begin{center}
\hfill
$
\begin{array}{cccc}
0 & 1 & 0 & 2\\
3 & 2 & 4 & 1\\
4 & 0 & 3 & 0\\
\end{array}
$
\hfill
$
\begin{array}{cccccc}
0 & 1 & 3 & 0 & 4 & 2\\
3 & 0 & 4 & 2 & 0 & 1\\
4 & 2 & 0 & 1 & 3 & 0\\
\end{array}
$
\hfill
~
\end{center}

Regardless of how these arrays are appended, vertices assigned rank $0$ form an independent set, and for each positive rank $t$, the vertices assigned rank $t$ form an independent set in $(C_3\squareop C_n)^2$.
\end{proof}

Since $C_3\squareop C_n$ has degeneracy $4$, it follows that $\trank(C_3\squareop C_n) \ge 5$ always.  When $n$ is odd, our lower bound improves.

\begin{lemma}\label{lem:lobd6}
If $n$ is odd, then $\trank(C_3 \squareop C_n)>5$.
\end{lemma}
\begin{proof}
Suppose for a contradiction that $C_3 \squareop C_n$ has a $2$-ranking $A$ using ranks in $[5]$.  Ranks $4$ and $5$ are \emph{high}; the other ranks are \emph{low}.  Note that each high rank appears at most once in every pair of adjacent columns of $A$.  It follows that at most $k$ vertices are assigned to each high rank.  A column containing all of the low ranks is \emph{low}, and a column containing all of the high ranks is \emph{high}.  Since $A$ has $2k+1$ columns and at most $2k$ vertices have high rank, it follows that some column of $A$ is \emph{low}.  

It is easy to check that $\trank(C_3\squareop P_2)\ge 5$.  It follows that a column adjacent to a low column must be high.  Since high ranks cannot appear in adjacent columns, a column adjacent to a high column must be low.  Therefore the columns of $A$ alternate high and low cyclically, contradicting that $n$ is odd.
\end{proof}

Collecting the lemmas, we obtain the following theorem.

\begin{theorem}
If $n$ is odd and $n\ge 25$, then $\trank(C_3\squareop C_n) = 6$.  If $n$ is even and $n\ge 6$, then $\trank(C_3\squareop C_n) = 5$.
\end{theorem}

It would be interesting to find the $2$-ranking number of $C_m\squareop C_n$ for general $m$ and $n$.  

\section*{Acknowledgements}
This research was supported in part by NSA grant H98230-14-1-0325.

\bigskip

\bibliographystyle{plain}
\bibliography{citations.bib}

\end{document}